\documentclass[12pt, twoside, reqno] {amsart}

\usepackage{amsfonts}
\usepackage{amssymb}
\usepackage{latexsym}
\usepackage{graphicx,graphics}
\usepackage{epsf,graphicx,latexsym%
}

\newcommand{\be} {\begin{eqnarray}}
\newcommand{\ee} {\end{eqnarray}}
\newcommand{\bep} {\begin{eqnarray*}}
\newcommand{\eep} {\end{eqnarray*}}

\newcommand {\m}{\mathop{\mathcal{M}}\nolimits}

\newcommand {\Hol}{\mathop{\rm Hol}\nolimits}

\newcommand {\Id}{\mathop{\rm Id}\nolimits}

\renewcommand {\Re}{\mathop{\rm Re}\nolimits}

\newcommand {\A}{\mathcal{A}}

\newcommand {\St}{\mathop{\mathcal{S}^*_e}\nolimits}
\newcommand {\Sta}{\mathop{\mathcal{S}^*(\B)}\nolimits}

\newcommand{\R}{{\mathbb R}}
\newcommand{\N}{{\mathbb N}}
\newcommand{\B}{{\mathbb B}}
\newcommand{\C}{{\mathbb C}}

\newcommand{\Z}{{\mathbb Z}}
\newcommand {\D}{\mathbb{D}}

\newtheorem{remar}{Remark}[section]
\newtheorem{examp}{Example}[section]
\newtheorem{defin}{Definition}[section]
\newtheorem{corol}{Corollary}[section]
\newtheorem{propo}{Proposition}[section]
\newtheorem{theorem}{Theorem}[section]
\newtheorem{lemma}{Lemma}[section]

\newcommand{\rema}{\begin{remar}\rm}
\newcommand{\erema}{$\blacktriangleright$\end{remar}}

\newcommand{\exa}{\begin{examp}\rm}
\newcommand{\eexa}{$\blacktriangleright$\end{examp}}

\def\lwvec(#1 #2){\linewd 0.1
           \lvec(#1 #2)
           \linewd 0.05}

\makeatletter
\@namedef{subjclassname@2020}{\textup{2020} Mathematics Subject Classification}
\makeatother

\begin{document}

\title[Multidimensional analogs of the Fekete--Szeg\"{o} functional]{Multidimensional analogs of the Fekete--Szeg\"{o} functional}

\begin{abstract} 
In this paper we  introduce the Fekete--Szeg\"{o} type mapping in the open unit ball of a complex Banach space. 
All previously studied modifications of the Fekete--Szeg\"{o} functional are either special cases or `components' of the mapping we introduce.

The study involves the examination of transforms of the Fekete--Szeg\"o mapping under specific transformations applied to given holomorphic mappings.

We show that for a  mapping $f$,  the third order Fréshet derivative of the inverse mapping $f^{-1}$ and of elements of the semigroup generated by $f$ can be expressed in terms of the Fekete--Szeg\"{o} mapping. Estimates of the Fekete--Szeg\"o mapping over some subclasses of semigroup generators and of starlike mappings are also presented.
\end{abstract}

\author[M. Elin]{Mark Elin}

\address{Department of Mathematics,
         Ort Braude College,
         Karmiel 21982,
         Israel}

\email{mark$\_$elin@braude.ac.il}

\author[F. Jacobzon]{Fiana Jacobzon}

\address{Department of Mathematics,
         Ort Braude College,
         Karmiel 21982,
         Israel}

\email{fiana@braude.ac.il}



\newenvironment{dedication}
  {
   \itshape             
  \raggedright         
  }
  {\par 
  }
\maketitle



\section{Introduction}\label{sect-intro}  
This paper focuses on multi-dimensional versions of the famous Fekete--Szeg\"o functional.  

\subsection{State of the art}
In their investigation  on the Bieberbach conjecture --- specifically in their disproof of Littlewood and Paley's conjecture --- Michael Fekete and Gabor Szeg\"o \cite{F-S} introduced the functional named after them and established its sharp estimate over the class of univalent normalized functions, refer to \cite{Dur} for details. This functional did not appear  accidentally, but because it articulates various geometric quantities in essential function transformations.
Additionally, it coincides with special cases of some other important functionals, such as the  Schwarzian derivative at zero,  Hankel's determinant and the generalized Zalcman's functional  (see \cite{E-V, Kanas, K-D, Ma-Mi, Zap-Tur} for details).
For these reasons, it has attracted keen attention of researchers in geometric function theory.  
In particular,  for ninety years, hundreds papers have derived estimates on the Fekete--Szeg\"o functional over diverse classes  of analytic functions (for some intriguing results see, e.g., \cite{C-K-S, E_J-23, Ke-Me, Koepf}).

As for the multidimensional  settings,  after Henri Cartan \cite{Cartan} showed that the Bieberbach conjecture fails in $\C^n$, interest in estimates of analogs of the Fekete--Szeg\"{o} functional waned. However, this interest has experienced a revival in the last decade. 

Several modifications to the classical Fekete--Szeg\"o functional have been proposed. Their estimates over various classes of holomorphic mappings were studied (see, for instance, \cite{D-L-22, E-J-22a, E-J-22b, Ham2023, HKK2021, HKK22} and \cite{XL14}--\cite{X-X-18}). 

\subsection{Goals of paper}
In this paper we consider holomorphic mappings in the open unit ball $\B$ of a complex Banach space $X$. Taking in mind the diversity of suggested modifications, our first purpose~is 

$\bullet\ $  {\it to  construct a unified  generalization of the Fekete--Szeg\"{o} functional that will include as  particular cases all analogs suggested earlier.} 

To do this, we  introduce the Fekete--Szeg\"{o} type mapping $f \mapsto\Psi_e(f,\lambda,\mu)$, $\lambda,\mu\in \C$, $e\in\partial \B$ in Definition~\ref{def-FS-oper}.
All previously studied modifications of the Fekete--Szeg\"{o} functional can be expressed in terms 
$\Psi_e(f,\lambda,\mu)$.

\vspace{2mm}
 Recall that in the one-dimensional case the  Fekete--Szeg\"o functional is of permanent interest in itself, including due to its  inherent  properties. However, geometric and analytical properties of its multi-dimensional analogs have not been studied at all. Thus the second and main goal of this work~is 

$\bullet\ $  {\it to establish properties of multi-dimensional generalizations of the Fekete--Szeg\"{o} functional.} 

It turns out that  properties of the mapping $\Psi_e(f,\lambda,\mu)$ are very similar to those of the original Fekete--Szeg\"{o} functional. Furthermore, a careful study of these properties shows that the faithful generalization of this  functional should be an $X$-valued mapping depending on two parameters.

\vspace{2mm}

Since we study a newly introduced correspondence $f\mapsto \Psi_e(f,\lambda,\mu)$, it is relevant 

$\bullet\ $  {\it to check its suitability by estimating $\Psi_e(\cdot, \lambda,\mu)$  over some broadly studied classes of mappings.} 

Results presented below essentially achieve our  goals.

\subsection{Main results and outlines} Let $\lambda,\mu\in \C$ and $e\in\partial \B$ be given.
As it is already mentioned, we introduce the Fekete--Szeg\"{o} mapping $f \mapsto\Psi_e(f,\lambda,\mu)$  in Definition~\ref{def-FS-oper} below. Its main properties can be formulated as follows.

\begin{theorem}\label{thm-main1}
  Let mapping $f $ be holomorphic in a neighborhood of the origin and represented there by the series of homogeneous polynomials $$f(x)=x+\sum\limits_{k=2}^\infty P_k(x).$$
\begin{itemize}
  \item [(i)] If  $P_2(e)=\nu e$, $\nu\in\C,$ and $g(x)=x+\sum\limits_{k=2}^\infty Q_k(x)$  is the $n$-th root transform of~$f$ defined by \eqref{sq-root}, then   $Q_{2n+1}(e)=\frac1n\, \Psi_e(f,\frac{n-1}{2n}, \mu)$ for any $\mu \in \C$.  In addition, if $f$ is of one-dimensional type and univalent in $\B$, then $g$ also is.
  \item [(ii)] If $X$ is a Hilbert space, $U$ is a unitary operator on $X$, and $g=U^*f\circ U$, then $\Psi_e(g,\lambda,\mu)= U^* \Psi_{\tilde{e}}(f,\lambda,\mu)$ with $\tilde{e}=Ue$.
  \item [(iii)] If the second derivative of either $f$ or $g$ at the origin is zero, then $\Psi_e( f\circ g,\lambda,\mu)  =  \Psi_e( f , \lambda, \mu)  +     \Psi_e( g, \lambda, \mu)$.
  \item [(iv)]If $f^{-1}$ denotes the inverse mapping of $f$, then \\ $\Psi_e(f^{-1},\lambda,\mu)=  -\Psi_e(f,2-\lambda,2-\mu)$. 
    \item [(v)] If $f^n$ denotes the $n$-th iterate of $f$, then $ \Psi_e(f^n,\lambda,\mu) =n \Psi_e(f,\lambda_n,\mu_n)$,  where $\lambda_n=n\lambda-n+1$ and $\mu_n=n\mu-n+1$.
 \end{itemize}  
 \end{theorem}

We prove these properties in Sections~\ref{sect-FSM} and~\ref{sect-compos}. Moreover, it turns out that the third Fr\'{e}chet derivative of the inverse mapping $f^{-1}$ and of each element of the semigroup generated by $f$ in a neighborhood of the origin are proportional to the Fekete--Szeg\"{o} mapping (see Theorems~\ref{thm-compo} and~\ref{th-semigr}).

As an additional task, we estimate the  ``error term'' $$R(f,g) =\Psi_e( f\circ g,\lambda,\mu)  -\bigl( \Psi_e( f , \lambda, \mu)  +     \Psi_e( g, \lambda, \mu) \bigr)$$ for the addition relation in general, cf. assertion (iii) in Theorem~\ref{thm-main1}. 

Section~\ref{sect-m} is devoted to the last goal formulated above. More precisely, we examine estimates of  the Fekete--Szeg\"{o} mapping $\Psi_e(\cdot, \lambda,\mu)$ over the class of semigroup generators $\m$, the associated class of starlike mappings and their special subclasses. 

It is worth mentioning that for a mapping $h\in \m$ and for the associated starlike mapping $f$, we prove that their Fekete--Szeg\"{o} mappings satisfy
\[
\Psi_e\left(h,2\lambda,2\mu\right)=-2\Psi_e(f,1-\lambda,1-\mu).
\] 

In order to present our results, we proceed with  some preliminary facts and notations in the next section.

 \medskip

\section{Preliminaries}\label{sect-preli}

Let $\D=\{z \in \C: |z|<1\}$ be  the open unit disk in the complex plane $\C$, and let a holomorphic in $\D$ function $f$ be represented by the Taylor series 
\begin{equation}\label{taylor}
  f(z)=z+\sum\limits_{k=2}^\infty a_kz^k.
\end{equation}

The {\it Fekete--Szeg\"o functional}  $f\mapsto\psi(f,\lambda)$ is defined by 
\begin{equation}\label{FS-func}
 \psi(f,\lambda):=a_3-\lambda a_2^2.
 \end{equation}

It was introduced in 1933 by Fekete and Szeg\"o \cite{F-S} whose proved that  {\it if function $f$ 
is univalent, then  
\[
|\psi(f,\lambda)| \le \left\{ \begin{array}{ll}
                               3-4\lambda, & \mbox{when } \lambda\le0,\vspace{2mm} \\
                                1+2\exp\left( \frac{-2\lambda}{1-\lambda} \right),  & \mbox{when }0\le \lambda\le1, \vspace{2mm}\\
                                4\lambda-3, & \mbox{when } \lambda\ge1 ,
                                                 \end{array}   \right.
\]
and this estimate is sharp.}   Particular attention is paid to the so-called  {\it Fekete--Szeg\"o problem}. For a given class $\mathcal{F}$ of analytic functions in $\D$, the Fekete--Szeg\"{o} problem over $\mathcal{F}$ is to find $\sup_{f\in\mathcal{F}} \left| \psi(f,\lambda)\right|$, see, for example, \cite{C-K-S, Ke-Me, Koepf}.

\vspace{2mm}

In the multidimensional settings,  several analogs of the Fekete--Szeg\"o functional have been suggested starting since 2014  (see, for instance, \cite{D-L-22, E-J-22b, Ham2023, HKK2021}, \cite{XL14}--\cite{X-L-Lu23}). To present them more specifically, we need the following notations.
\vspace{2mm}

Let $X$ be a complex Banach space
, $L(X)$ be the space of all bounded linear operators on $X$, and $X^*$ be the dual space of $X$ with pairing $\langle x,a\rangle,\ x\in X, a\in X^*$. We denote by $\B$ the open unit ball in $X$. For each $x \in X$ denote
\begin{equation*}\label{Tx-set}
J(x)=\left\{x^* \in X^*: \|x^*\|=\|x\| \text{ and } \langle x,x^*\rangle=\|x\|^2 \right\}.
\end{equation*}
According to the Hahn--Banach theorem (see, for example, \cite[Theorem~3.2]{Rudin}), the set $J(x)$ is nonempty and may consist of a singleton (for instance, when $X$ is a Hilbert space), or, otherwise, of infinitely many elements. Its elements $x^* \in J(x)$ are called support functionals at the point $x$. 
 
Let $Y$ be another complex Banach space. A mapping ${f:\B\to Y}$  is called holomorphic if it is Fr\'{e}chet differentiable at every point $x_0\in\B$. The set of all holomorphic mappings from $\B$ into $Y$ will be denoted by $\Hol(\B,Y)$. It is well known (see, for example, \cite{E-R-S-19, GK2003, R-Sbook}) that if $f \in \Hol(\B,Y)$, then for every $x_0\in\B$ and  all $x$ in some neighborhood of $x_0 \in \B$, the mapping $f$ can be represented by the series of homogeneous polynomials 
\begin{equation}\label{Taylor-infinite}
f(x) =\sum_{n=0}^{\infty} P_n(x),\quad P_n(x)=\frac{1}{n!} D^nf(x_0)\left[(x - x_0)^n\right],
\end{equation}
where $D^nf(x_0):\prod\limits_{k=1}^{n}X \to Y$ is a bounded symmetric $n$-linear operator which is called the $n$-th  Fr\'{e}chet derivative of $f$ at $x_0$. 
One says that $f\in\Hol(\B,X)$ is {\it normalized} if $f(0)=0$ and $Df(0)=\Id,$ where $\Id\in L(X)$ is the identity operator on $X$. We denote by $\A(\B)$ the subset of $\Hol(\B,X)$ consisting of all normalized mappings.

Another important subclass of $\Hol(\B,X)$ consists of mappings having the form $f(x)=s(x)x,\ s\in\Hol (\B,\C)$. We say that such mappings are \textit{of one-dimensional type}. This subclass has been studied by many mathematicians (see, for example, \cite{D-L-21, ES2004, Lic-86} and references therein). In this connection, the following fact is helpful (for detailed calculation see \cite{E-J-22a}). 
\begin{lemma}\label{lemm-F-S-norm1}
Let $f$  be represented by the series \eqref{Taylor-infinite}.  If $f$ is of one-dimensional type, then for every $j\in\N$, the homogeneous polynomial $P_j$ is  also of one-dimensional type. 
\end{lemma}


 A very natural analog of the Fekete--Szeg\"o functional was suggested in~\cite{XL14}. Fix arbitrary $e\in\partial\B$ and $e^*\in J(e)$ and then denote 
\begin{equation}\label{FS-func1}
 \psi_e^{(1)}(f,\lambda):= a_3  - \lambda \left(a_2 \right) ^2,
 \end{equation}
 where
 \begin{equation}\label{aj}
 a_j := \left\langle P_j(e),e^* \right\rangle \quad\mbox{with}\quad P_j(e):=\frac1{j!}D^jf(0)[ e^j ].
 \end{equation}
The estimation of this functional over various classes of mappings (such as starlike, strongly starlike of order $\alpha$, close-to-convex of type B and so on) has been studied in many works (see, for instance, \cite{D-L-22, La-X-21}, \cite{X-L-L-18}--\cite{X-X-18}), where holomorphic mappings $f$ represented by \eqref{Taylor-infinite} were subject to the additional restriction: 
\begin{equation}\label{addi_condi}
\frac12 \left\langle D^2f(0)\left[ x, P_2(x) \right]  ,x^* \right\rangle = \left\langle P_2(x) ,x^* \right\rangle^2,\quad x\in\partial\B,\ x^*\in J(x).
\end{equation}
Note that this condition is obviously satisfied for mappings for which the second Fr\'{e}chet derivative at the origin is of one-dimensional type, see \cite{HKK2022}. It should also be noted that most of the work before 2020 considered mappings of one-dimensional type, or at least assumed condition \eqref{addi_condi} to be hold. 

In \cite{E-J-22a}, the authors obtained sharp estimates on  the norm of the mapping
\begin{equation}\label{FS-oper1}
 \Psi_e^{(1)}(f,\lambda):= P_3  -\frac\lambda2 D^2f(0)[e, P_2(e)],
 \end{equation}
 while the mapping $f$ was yet assumed to be of one-dimensional type.

To  explore wider classes of mappings, the following functionals were introduced in \cite{HKK2021} (see also \cite{HKK22}):
\begin{eqnarray}\label{FS-func2}
\begin{array}{ll}
   \psi_e^{(2)}(f,\lambda)  :=& a_3  - (\lambda-1) \left(a_2 \right) ^2 - \widetilde{a}_2^2, \vspace{1mm}\\
     \psi_e^{(3)}(f,\lambda) :=& a_3  - \left(\lambda-\frac23\right) \left(a_2 \right) ^2 - \frac23\,\widetilde{a}_2^2, \vspace{1mm} \\
     \psi_e^{(4)}(f,\lambda) :=& a_3  - (\lambda-2) \left(a_2 \right) ^2 -2\, \widetilde{a}_2^2,
     \end{array}
\end{eqnarray}
where $a_2,a_3$ are defined by~\eqref{aj} and
\begin{equation*}\label{a_2^2}
 \widetilde{a_2^2}:= \frac12 \left\langle D^2f(0) \left[e, P_2(e) \right] , e^* \right\rangle.
\end{equation*}
There were established sharp estimates on $ \psi_e^{(2)}$ over the class of starlike mappings, on $\psi_e^{(3)}$  over  the class of quasiconvex mappings of type B, and on  $\psi_e^{(4)}$  over  the class of normalized resolvents associated with normalized generators (for this class see \cite{E-R-S-19, GK2003}).

Note that the functionals $\psi_e^{(1)}$\!, $\psi_e^{(2)}$\!,  $\psi_e^{(3)}$  and  $\psi_e^{(4)}$  coincide under condition~\eqref{addi_condi}. 
 
In addition, the functional $\psi_e^{(2)}$  works successfully when we consider spirallike mappings relative to a scalar operator ${A=e^{i\beta}\Id}$. Unfortunately, $\psi_e^{(2)}$  is not suitable when $A$ is not scalar. 
In \cite{E-J-22b}, the authors suggested a modification of $\psi_e^{(2)}$  which enables to involve spirallike mappings relative to any strongly accretive operator $A\in L(X)$. Namely, we denoted $ \psi_e^{A}(f,\lambda)=a_3^A  - (\lambda-1) \left(a_2^A \right) ^2 - \widetilde{a}_2^{2,A}$, where 
 \begin{equation*}\label{a223-renewed}
   \begin{array}{cc}
     a_2^A =& \left\langle D^2f(0)[e,Ae] - AP_2(e),e^*\right\rangle, \hfill \vspace{2mm} \\
     a_3^A =&\frac14 \left\langle D^3f(0)[e^2,Ae] - 2AP_3(e),e^*\right\rangle , \vspace{2mm} \hfill \\
     \widetilde{a}_2^{2,A}  =& \frac12 \left\langle D^2f(0)\left[e,D^2f(0)[e,Ae]\right] - D^2f(0)[e,AP_2(x)]  , e^*\right\rangle .
   \end{array}
 \end{equation*}
 Obviously, in the case of a scalar operator $A$  the functionals $\psi_e^{(2)}$ and   $\psi_e^{A}$  coincide up to multiplication by scalar.
 
\vspace{2mm}
 
In constructions below we use the second Fr\'{e}chet derivative which is  a bi-linear symmetric operator. Therefore the following assertion will be helpful.

\begin{lemma}[Polarization identity]\label{lem-sym}
Each bi-linear symmetric operator $S$ satisfies 
\begin{equation}\label{symmet}
  S[x_1,x_2] = \frac14\left( S[(x_1+x_2)^2]  - S[(x_1-x_2)^2] \right).
\end{equation}  
\end{lemma}

\medskip

\section{Fekete--Szeg\"o mapping and its properties}\label{sect-FSM}
\setcounter{equation}{0}


Let $f $ be represented by the series of homogeneous polynomials 
 \begin{equation}\label{series}
   f(x)=x+\sum\limits_{k=2}^\infty P_k(x), 
 \end{equation} 
 which converges in a neighborhood of the origin in $X$, and hence $f$ is holomorphic in this neighborhood. 
 
 We introduce the generalization of the Fekete--Szeg\"o functional, which is the main object of study in this paper.

\begin{defin}\label{def-FS-oper}
Let $f $ be defined by \eqref{series}.  For given $\lambda,\mu\in\C,$ $e \in \partial \B$  and $e^*\in J(e)$, the mapping
\begin{eqnarray*}
   \Psi_e(f,\lambda,\mu) = P_3(e)       - \frac\mu2 D^2f(0)\left[ e, P_2(e) \right] 
           -  (\lambda-\mu) \left\langle P_2(e),e^* \right\rangle P_2(e) 
  \end{eqnarray*}
  will be called  the Fekete--Szeg\"o mapping  for $f$.
 \end{defin}
 
 As for the above-mentioned generalizations of the Fekete--Szeg\"o functional  (see  \eqref{FS-func1}, \eqref{FS-func2} and \eqref{FS-oper1}), they can be expressed in terms $\Psi_e(f,\lambda,\mu)$, namely,  
  \begin{eqnarray}\label{FS-previ}
  \begin{array}{ll}
     \psi_e^{(1)}(f,\lambda)   =&  \left\langle \Psi_e(f,\lambda,0) ,e^*\right\rangle,     \\
          \psi_e^{(2)}(f,\lambda)   =&  \left\langle \Psi_e(f,\lambda,1) ,e^*\right\rangle,  \\
          \psi_e^{(3)}(f,\lambda)   =&  \left\langle \Psi_e\left(f,\lambda, 2/3\right) ,e^*\right\rangle,  \\
           \psi_e^{(4)}(f,\lambda)  =&  \left\langle \Psi_e\left(f,\lambda, 2\right) ,e^*\right\rangle,      \\
          \Psi_e^{(1)}(f,\lambda)   =& \Psi_e(f,\lambda,\lambda).
  \end{array}
 \end{eqnarray}

\vspace{2mm}

Throughout this paper $\lambda,\mu\in\C,$ $e \in \partial \B$  and $e^*\in J(e)$ assumed to be given.
 From now until the end of this section we also assume that $f$ is holomorphic in the open unit ball $\B$.  
 
First we point out specific properties of the Fekete--Szeg\"o mapping over one-dimensional type mappings.
\begin{lemma}\label{lemm-F-S-norm}
Let $f\in\A(\B)$  be represented by the series \eqref{series}.
\begin{itemize}
  \item [(A)] If $P_2$ is of one-dimensional type, that is,  $P_2(x)= s(x)x,$  where $s$ is a homogeneous function, then  $\Psi_e(f,\lambda,\mu)$ does not depend on $\mu$, so $\Psi_e(f,\lambda,\mu) = P_3(e) - \lambda \left\langle P_2(e),e^* \right\rangle P_2(e) $.
  \item [(B)] If both $P_2$  and $P_3$ are of one-dimensional type then 
  \[
     \Psi_e(f,\lambda,\mu)  =  \left\langle  \Psi_e(f,\lambda,\mu), e^*\right\rangle e .    
      \]
\end{itemize}
\end{lemma}

\begin{proof} 
If $P_2(x)= s(x)x,$  then $\left\langle P_2(e),e^* \right\rangle = s (e)$ and hence 
\[
\frac12 D^2f(0)\left[ e, P_2(e) \right] =\frac{s(e)}2 D^2f(0)\left[ e^2 \right] = s(e) P_2(e) = \left\langle P_2(e),e^* \right\rangle P_2(e).
\]
So, assertion (A) follows. 

Similarly, if $P_2$  and $P_3$ are of one-dimensional type then  $\Psi_e(f,\lambda,\mu) =\tilde{s}(e)e$, where $\tilde{s}$ is a scalar-valued polynomial. 
 This proves assertion (B).
\end{proof}

Now we turn on to discuss transformation properties of the Fekete--Szeg\"o mapping $f\mapsto \Psi_e(f,\lambda,\mu)$.

\subsection{Generalized $n$-th root transform}

In the one-dimensional case the $n$-th root transform of $f\in\A(\D)$ is defined by
\[
g(z):=\sqrt[n]{f(z^n)} :=\sqrt[n]{p(z^n)}\,z,\quad \mbox{where }\ p(z)=\frac{f(z)}{z},
\]
and the branch of the root is chosen such that $\sqrt[n]{p(0)}=1$.

If $f$ has the Taylor expansion \eqref{taylor}, then $g(z)=z+\sum\limits_{m=1}^\infty b_{nm+1}z^{mn+1}$  with $b_{n+1}=\frac{1}{n}a_2$ and $b_{2n+1}=\frac1n\, \psi(f,\frac{n-1}{2n})$, where $\psi$ is defined by \eqref{FS-func}. 

\vspace{2mm}

To discuss possible infinite-dimensional analogs of these relations, note that for $x \in \B$ one has 
\begin{eqnarray*}
   f(x)  &=& f(x)- f(0)=\int_0^1 \frac{d}{dt}f(tx) dt \\
     &=&  \int_0^1 Df(tx)xdt =\left( \int_0^1 Df(tx) dt  \right) x. 
  \end{eqnarray*}
  Thus the following auxiliary assertion holds:
\begin{lemma}\label{lem-0}
  Let $f\in\A(\B)$. Then there exists a (not unique) mapping $ \kappa\in\Hol(\B,L(X))$ such that $f(x)=\kappa(x)x$ and $\kappa(0)=\Id$.  
\end{lemma}

Given any $f\in\A(\B)$ and some appropriate $\kappa$, there is $\rho\in(0,1]$ such that the union of the numerical ranges
\[
\bigcup_{\|x\|<\rho} \left\{\left\langle\kappa(x) y,y^* \right\rangle,\ \|y\|=1,\ y^*\in J(y)  \right\}
\]
does not separate zero and infinity. Therefore the powers $\kappa^\alpha(x)$, $\alpha>0,$ are well-defined holomorphic mappings 
on the ball $\|x\|<\rho$ and satisfy  $\kappa^\alpha(0)=\Id$. Sometimes $\rho$  can be chosen to be $1$. It happens, for example, when 
(i)~$X$ is a Hilbert space and $\Re\kappa(x)>0$ for all $x\in\B$, when (ii) $\kappa$ is a scalar operator, $\kappa(x)=s(x)\Id$   and $s(\B)$  does not separate zero and infinity in $\C$, and in many other cases.

Let $e\in \partial \B$ and $f\in\A(\B)$ be represented as in Lemma~\ref{lem-0}. The mapping
\begin{eqnarray}\label{sq-root}
g(x):=\left( \kappa\left(\langle x,e^*\rangle^n e\right)\right)^{\frac1n}\,x, \qquad x\in\B_\rho,
\end{eqnarray}
will be called the $n$-th root transform of $f$. Although $g$ depends on the choice of~$\kappa$, every $n$-th root transform of $f$ has the following properties.

\begin{theorem}\label{thm-FS-th}
Let $f\in\A(\B)$ be represented by \eqref{series} and  have the form $f(x)=\kappa(x)x$ with $\kappa\in\Hol(\B, L(x))$. Assume that $P_2(e)=\nu e$  for some ${\nu \in \C}$. Let $g$ be defined by \eqref{sq-root} and represented by the series of homogeneous polynomials $g(x)=x+\sum\limits_{k=2}^\infty Q_k(x)$.

Then $Q_{n+1}(e)=\frac{1}{n}P_2(e)$, $Q_{2n+1}(e)=\frac1n\, \Psi_e(f,\frac{n-1}{2n}, \mu)$  for any $\mu \in \C$, and  $Q_k(x)=0$ whenever $k\neq mn+1,\ m\in\N$.
\end{theorem}

\begin{proof}
Consider the series of homogeneous polynomials of $\kappa$:
\begin{equation*}
\kappa(x)=\Id+\sum_{k=1}^{\infty}p_k(x), \quad\text{so that}\quad p_k\in\Hol(\B,L(X)),\  P_{k+1}(x)=p_k(x)x.
\end{equation*}
Denote $A_k=p_k(e)\in L(X)$. Then
\begin{equation*}
\kappa\left(\langle x,e^*\rangle^n e\right)=\Id + \sum_{k=1}^{\infty} p_k\left(\langle x,e^*\rangle^n e\right) 
= \Id+\sum_{k=1}^{\infty}  \langle x,e^*\rangle^{nk} A_k.
\end{equation*}

Using the binomial series we get
\begin{eqnarray*}\label{g-calc}
g(x)&:=& \left(  \Id+\sum_{k=1}^{\infty}  \langle x,e^*\rangle^{nk} A_k   \right) ^{\frac1n} \, x \\
\nonumber &=& \left[ \Id+ \sum_{m=1}^\infty \left( \begin{array}{c}
                                                   1/n \\
                                                   m
                                                 \end{array}
\right) \left( \sum_{k=1}^{\infty}  \langle x,e^*\rangle^{nk} A_k \right)^{\!m}    \right] x\\
\nonumber &=&  \left[ \Id + \frac1n \sum_{k=1}^{\infty}  \langle x,e^*\rangle^{nk}  A_k + \frac{1-n}{2n^2} \left( \sum_{k=1}^{\infty}
 \langle x,e^*\rangle^{nk}  A_k \right)^2+  \ldots      \right] x                                \\
\nonumber &=& x+\frac{1}{n}  \langle x,e^*\rangle^{n} A_1 x+\frac{1}{n}  \langle x,e^*\rangle^{2n} A_2 x + \frac{1-n}{2n}
 \langle x,e^*\rangle^{2n} A_1^2 x + \ldots.
\end{eqnarray*}
This calculation shows that $Q_k(x)=0$ whenever $k\neq mn+1,\ m\in\N$, while
\begin{eqnarray*}\label{g-tay}
&&Q_1(x)=x,\quad Q_{n+1}(x)=\frac{1}{n}   \langle x,e^*\rangle^{n} A_1 x,\\
\nonumber &&Q_{2n+1}(x)=\frac{1}{n} \left(A_2  + \frac{1-n}{2n} A_1^2\right)  \langle x,e^*\rangle^{2k}  x.
\end{eqnarray*}
Thus  the result follows by Definition~\ref{def-FS-oper}.  
\end{proof}

The root transform in the one-dimensional case is interesting because the univalence of $f\in \A(\D)$ implies the univalence of its root transform $g$. For the multi-dimensional case, we do not know whether this implication remains true.  We prove this implication for mappings of one-dimensional type.  
\begin{theorem}\label{thm-root-tr-univ}
Let 
 $f(x)=s(x)x$, $s\in \Hol(\B,\C)$  and $s(\B)$  does not separate zero and infinity.
Let $g$ be the $n$-th root transform of $f$ defined by 
\begin{equation*}
g(x)=\left( s\left(\langle x,e^*\rangle^n e\right)\right)^{\frac1n}\,x.
\end{equation*}
If $f$ is univalent in $\B$, then $g$ also is.
\end{theorem}
\begin{proof}
It follows from the discussion after Lemma~\ref{lem-0}  that $g$ is holomorphic in $\B$. 
Take $x_1,\, x_2 \in \B$ and assume that $g(x_1)=g( x_2) $, then
\begin{eqnarray*}
\left( s\left(\langle x_1,e^*\rangle^n e\right)\right)^{\frac1n}\,x_1 &=& \left( s\left(\langle x_2,e^*\rangle^n e\right)\right)^{\frac1n}\,x_2. 
\end{eqnarray*}
Thus there exists $\gamma \in \C$ such that $x_2=\gamma x_1$. Hence,
\begin{eqnarray*}
\left( s\left(\langle x_1,e^*\rangle^n e\right)\right)^{\frac1n}\,\langle x_1,e^*\rangle  &=& \left( s\left(\langle \gamma x_1,e^*\rangle^n e\right)\right)^{\frac1n}\,\gamma \langle x_1, e^*\rangle . 
\end{eqnarray*}
Raising both sides to the $n$-th power, we conclude that
\begin{eqnarray*}
 s\left(\langle x_1,e^*\rangle^n e\right)\,  \langle x_1,e^*\rangle ^n e =  s\left(\langle \gamma x_1,e^*\rangle^n e\right)\,\langle \gamma x_1, e^*\rangle^n e,
\end{eqnarray*}
which is equivalent to $f(\langle x_1,e^*\rangle e)=f(\langle \gamma x_1,e^*\rangle e)$. Since $f$ is univalent, then $\langle x_1,e^*\rangle =\gamma \langle  x_1,e^*\rangle$. 

This means that either $\langle x_1,e^*\rangle=0$ or $\gamma=1$.
If $\langle x_1,e^*\rangle=0$  then  
$\left( s\left(0\right)\right)^\frac{1}{n}\,x_1 = \left( s\left(0\right)\right)^\frac{1}{n}\,\gamma x_1,$  with   $s\left(0\right)=1.$ Therefore, in both cases $x_1=x_2$, and the proof is complete.
\end{proof}

\vspace{2mm}

\subsection{Unitary transform}
Once again let begin with  $f\in \A(\D)$ represented by \eqref{taylor}. Given  some ${\theta \in \R}$, denote $u=\exp\left(i \theta \right)$. The rotation transform of $f$ is defined by $g(z):=\overline{u} f(uz)$.  Writing $g(z)=z+\sum\limits_{n=2}^{\infty}b_n z^n$, we get $b_n=a_nu^{n-1}$,  $n \in \N$, and
\[
\psi(g,\lambda)=u^2 \psi(f,\lambda),
\]
where $\psi$ is the Fekete--Szeg\"o functional defined by \eqref{FS-func}.

\vspace{2mm}

Let  $X$ be  a Hilbert space. Then $e^*\in J(e)$ can be identified with $e$.   For a given unitary operator  $U \in L(X)$, the unitary transform of $f$ is defined  by
\begin{equation*}
g=U^*f\circ U.
\end{equation*}

\begin{theorem}\label{thm-rotation}
  Let $f \in \A(\B)$ and $g$ be its unitary transform. \\ Then  $\Psi_e(g,\lambda,\mu)= U^* \Psi_{\tilde{e}}(f,\lambda,\mu)$ with $\tilde{e}=Ue$.
\end{theorem}

\begin{proof}
According to the definition of the unitary transform 
\begin{eqnarray*}
g(x)= U^*\left(Ux+\sum\limits_{k=2}^\infty P_k(Ux)\right)=x+\sum\limits_{k=2}^\infty U^*P_k(Ux).
\end{eqnarray*}
Then for any $x\in\partial\B$ and $k\geq 2$ we have
\begin{eqnarray}\label{deriv}
\frac{1}{k!}D^k g(0)[x^k]=U^*P_k(Ux).
\end{eqnarray}

Let us now calculate $\Psi_e(g,\lambda,\mu)$: 
  \begin{eqnarray*}
    U\Psi_e(g,\lambda,\mu)=P_3(Ue) - \frac{\mu}2 U D^2g(0)\left[ e, U^*P_2(Ue)\right]   \\
            - (\lambda-\mu)\left\langle P_2(Ue) , Ue \right\rangle P_2(Ue) .
  \end{eqnarray*}

  In addition, formulas \eqref{deriv} and \eqref{symmet}  imply  
  \[
   D^2g(0)\left[ e, U^*P_2(Ue) \right] = \frac12 U^* \left[P_2(Ue + P_2(Ue)  ) -P_2(Ue - P_2(Ue)  )  \right].
  \]

 Denoting $\tilde{e}=Ue$ 
  we have
   \begin{eqnarray*}
    U\Psi_e(g,\lambda,\mu) &=& P_3(\tilde{e}) -(\lambda-\mu) \left\langle P_2(\tilde{e}) , \tilde{e} \right\rangle P_2(\tilde{e}) \\
     &-& \frac{\mu}4 \left[P_2(\tilde{e} + P_2(\tilde{e})  ) -P_2(\tilde{e} - P_2(\tilde{e})  )  \right] \\
     &=& P_3(\tilde{e}) - (\lambda-\mu)\left\langle P_2(\tilde{e}) , \tilde{e} \right\rangle P_2(\tilde{e})   - \frac{\mu}2 D^2f(0)\left[\tilde{e},  P_2(\tilde{e}) \right],
  \end{eqnarray*}
as desired.
\end{proof}

\medskip

\section{Composition of mappings and inverse mappings}\label{sect-compos}
\setcounter{equation}{0}

Let two functions $f$ and $g$ be holomorphic in a neighborhood of $0\in\C$ and represented by the Taylor series $f(z)=z+\sum\limits_{n=2}^{\infty}a_n z^n$ and $g(z)=z+\sum\limits_{n=2}^{\infty}b_n z^n$. Assume that $(f \circ g)(z)  =z+\sum\limits_{n=2}^{\infty}c_n z^n$. Then we have
\begin{equation*}
c_2=a_2+b_2, \quad c_3= a_3 +2a_2b_2+b_3,
\end{equation*}
and
\begin{equation}\label{psi-compo-1dim}
\psi(f\circ g,\lambda)= \psi(f,\lambda) + \psi(g,\lambda) - 2(\lambda-1)a_2b_2,
\end{equation}
where $\psi$ is the Fekete--Szeg\"o functional defined by \eqref{FS-func}. Moreover, if $g=f^{-1}$, then $c_2=c_3=0,$ which implies
\begin{equation*}
b_2=-a_2, \quad b_3=-(a_3-2a_2^2)=-\psi(f,2).
\end{equation*}
In turn, it follows from these relations that
\begin{equation*}
\psi(g,\lambda)=-\psi(f,2-\lambda).
\end{equation*}

\vspace{2mm}
In this section we assume that $f$ and $g $ are holomorphic in a neighborhood of the origin in a Banach space $X$ and represented there by the series of homogeneous polynomials 
\begin{equation}\label{series-2}
f(x)=x+  \sum\limits_{k=2}^\infty P_k(x)\quad\mbox{and}\quad       
g(x)=x+\sum\limits_{k=2}^\infty Q_k(x).
\end{equation}

\subsection{Composition formula and its applications}
This subsection is devoted to infinite-dimensional analogs of \eqref{psi-compo-1dim} and some close relations. They are based on the next result.
\begin{theorem}\label{lemm-compos}
 Let $f$ and $g$ be represented by the series \eqref{series-2}. Then 
\begin{eqnarray}\label{Psi-compo}
 \nonumber   \Psi_e(f\circ g,\lambda,\mu)  &\!\!=\!\!&  \Psi_e( f , \lambda, \mu)  + 
                       \Psi_e(  g, \lambda, \mu) \\
\nonumber  &\!\!-\!\!&   (\lambda-\mu) \left[ \left\langle P_2(e) ,e^* \right\rangle        Q_2(e)  + \left\langle Q_2(e),e^* \right\rangle P_2(e) \right]    \\
&\!\!-\!\!& \mu C[e, P_2(e)]   -  (\mu-2) B[e,Q_2(e)] ,
\end{eqnarray}
 where $B:=\frac12 D^2f(0)$ and  $C:=\frac12 D^2g(0)$.
\end{theorem}

\begin{proof}
Take $x$ close enough to the origin. Then by \eqref{series-2} we get 
\begin{eqnarray*}\label{compos}
&& (f\circ g)(x)= \\
&\!\!=\!\!& \left(x+Q_2(x)+Q_3(x)+P_2(x+Q_2(x))+P_3(x)+\ldots\right)\\
&\!\!=\!\!&x+ \left(Q_2(x)+Q_3(x)+B[(x+Q_2(x))^2]+P_3(x)+\ldots\right)\\
&\!\!=\!\!& x\!+\! \left\{Q_2(x) \!+\! P_2(x)\right\}\!+\! \left\{Q_3(x) \!+\! 2B[x,Q_2(x)] \!+\! P_3(x)\right\}\!+\ldots.
\end{eqnarray*}
 
Denoting $R_k(x):=\frac{1}{k!} D^k (f\circ g)(0)$, one concludes
\begin{equation}\label{R2R3}
R_2(x)= P_2(x)+ Q_2(x), \ R_3(x)=2 B[x, Q_2(x)]+ P_3(x)+ Q_3(x).
\end{equation}

Let us use the polarization identity~\eqref{symmet} to calculate the Fekete--Szeg\"o mapping:
\begin{eqnarray*}
   \Psi_e( f\circ g,\lambda,\mu)  &=&  2B[e, Q_2(e)]+ P_3(e)+ Q_3(e)  \\
   &-& (\lambda-\mu) \left\{  \left\langle P_2(e),e^* \right\rangle P_2(e) +
   \left\langle P_2(e) ,e^* \right\rangle Q_2(e) \right. \\
   &+&\left.  \left\langle Q_2(e),e^* \right\rangle P_2(e) +
    \left\langle Q_2(e),e^* \right\rangle Q_2(e) \right\}   \\
   &-&  \mu \left( B[e, P_2(e)] + C[e, P_2(e)]  + B[e, Q_2(e)] + C[e, Q_2(e)] \right) ,
\end{eqnarray*}
which implies \eqref{Psi-compo}.
\end{proof}

Note that in the one-dimensional case if $a_2b_2=0$ then
$\psi(f\circ g,\lambda)= \psi(f,\lambda) + \psi(g,\lambda)$. In the multi-dimensional settings, Theorem~\ref{lemm-compos} implies
\begin{corol}[Addition formula]\label{corol-odd}
If either $P_2 = 0$  or $Q_2= 0$, then 
\[
\Psi_e( f\circ g,\lambda,\mu)  =  \Psi_e( f , \lambda, \mu)  +     \Psi_e( g, \lambda, \mu).
\]
\end{corol}

Another consequence of Theorem~\ref{lemm-compos} and formulas~\eqref{FS-func1}, \eqref{FS-func2} and~\eqref{FS-oper1}~is
\begin{corol}
 \begin{eqnarray*}
          \psi_e^{(1)}(f\circ g,\lambda)   &=&   \psi_e^{(0)}(f,\lambda) + \psi_e^{(0)}(g,\lambda) + 2 \left\langle B_2[e,Q_2(e)], e^*\right\rangle \\
         && -  2\lambda  \left\langle P_2(e) ,e^* \right\rangle \left\langle Q_2(e),e^*\right\rangle ,      \\
          \psi_e^{(2)}(f\circ g,\lambda)   &=&   \psi_e^{(2)}(f,\lambda) + \psi_e^{(2)}(g,\lambda) \\
        &&  + \left\langle  B_2[e,Q_2(e)]   -  C_2[e,P_2(e)], e^*\right\rangle  \\
        &&  - 2 (\lambda-1) \left\langle P_2(e) ,e^* \right\rangle \left\langle Q_2(e),e^*\right\rangle  ,  \\
          \Psi_e^{(1)}(f\circ g,\lambda)   &=& \Psi_e^{(1)}(f,\lambda) + \Psi_e^{(1)}(g,\lambda) \\
        && - \lambda C_2[e,P_2(e)] + (2-\lambda)B_2[e,Q_2(e)] .
 \end{eqnarray*} 
\end{corol}

Observe that for mappings of one-dimensional type formula~\eqref{Psi-compo} takes an essentially simpler form.
\begin{corol}[cf.~formula~\eqref{psi-compo-1dim}]\label{corol-add}
Let $f$ and $g $ be represented by the series  of homogeneous polynomials
$f(x)=x+  \sum_{k\ge2}p_k(x)x$ and $g(x)=x+\sum_{k\ge2}q_k(x)x $, where $p_k,q_k \in \Hol(\B,\C)$,  $k \in \N$.
Then 
\begin{eqnarray*}\label{Psi-comp-1dim-fg}
&&\Psi_e( f\circ g,\lambda,\mu) =  \Psi_e( f, \lambda, \mu)  + 
                       \Psi_e( g, \lambda, \mu) -2( \lambda-1)p_2(e) q_2(e)e.
\end{eqnarray*}
  \end{corol}
\begin{proof}
Formula~\eqref{Psi-compo}  implies
  \begin{eqnarray*}\label{Psi-compo-1dim1-1}
&&\Psi_e( f\circ g,\lambda,\mu) =  \Psi_e( f, \lambda, \mu)  + 
                       \Psi_e( g, \lambda, \mu) \\
&&- (\lambda-\mu)q_2(e)\Bigl\{  \left\langle p_2(e)e ,e^* \right\rangle e
+ p_2(e)e\Bigr\}     \\
 && -\mu  C_2[e,p_2(e)e]   -   q_2(e)(\mu-2) p_2(e)e,
\end{eqnarray*}
which is equivalent to the claim.
\end{proof}

Next,  consider some  applications of Theorem~\ref{lemm-compos}.  Since $f$ is  holomorphic in a neighborhood of the origin,  for every $n\in \N$ its iterate  $f^n:=\underbrace{ f\circ f\circ\ldots\circ f}\limits_{n\ \scriptstyle{times}}$,  is well-defined in a (perhaps, smaller) neighborhood. Therefore one can represent $f^n$ by a series of homogeneous polynomials and hence calculate $\Psi_e(f^n,\lambda,\mu)$.

\begin{propo}\label{thm-itera}
Let  $f$ be represented by the series \eqref{series-2}. For $n\in\N$  denote $T_{n,j}(x):=\frac{1}{j!}D^j(f^n)(0)[x^j]$. Then
  \begin{eqnarray*}
    T_{n,2}(x) &=& n P_2(x),  \\
    T_{n,3} (x) &=& n P_3(x) +(n^2-n) B[x,P_2(x)],  \\
    \Psi_e(f^n,\lambda,\mu) &=&n \Psi_e(f,\lambda_n,\mu_n),
  \end{eqnarray*}
  where $\lambda_n=n\lambda-n+1$ and $\mu_n=n\mu-n+1$. (As above, we denote $B=\frac12D^2f(0)$.)
\end{propo}

\begin{proof}
  We proof this theorem by induction. Obviously, the statement holds for $n=1$. Assume the claim holds for some $n=k$ and consider ${f^{k+1}=f\circ f^k.}$ In order to use \eqref{R2R3}, let take $g=f^k$. Then $T_{n,2} (x)=P_2(x) +kP_2(x)=(k+1)P_2(x)$ and
  \begin{eqnarray*}
    T_{n,3} (x) &=& 2B[x,kP_2(x)]+P_3(x) + k P_3(x) +(k^2-k) B[x,P_2(x)] \\
     &=& (k+1)P_3(x) + (2k+k^2-k) B[x, P_2(x)]  .
  \end{eqnarray*}
  So, it remains to prove the last equality in the theorem. Applying Theorem~\ref{lemm-compos} to $g=f^k$ again, we get
   \begin{eqnarray*} 
  \Psi_e(f^{k+1},\lambda,\mu)  &=&   \Psi_e(f,\lambda,\mu) +k \Psi_e(f,\lambda_k,\mu_k) \\
       && -(\lambda-\mu) \left\{  \left\langle P_2(e) ,e^* \right\rangle kP_2(e) + \left\langle kP_2(e),e^* \right\rangle P_2(e) \right\}   \\
        &&-  \mu kB[e,P_2(e)] + (2-\mu)B[e,kP_2(e)] \\
        &=&  P_3(e) - (\lambda-\mu) \left\langle P_2(e),e^* \right\rangle P_2(e) - \mu B[ e, P_2(e)  ]    \\ 
        &&   + k\bigl( P_3(e) - (\lambda_k-\mu_k) \left\langle P_2(e),e^* \right\rangle P_2(e) - \mu_k B\left[ e, P_2(e) \right] \bigr)   \\
       && -(\lambda-\mu)2k  \left\langle P_2(e) ,e^* \right\rangle P_2(e) - 2k(\mu-1) B[e,P_2(e)]    \\
       &=& (k+1)P_3(e)  - \left( (2k+1)\mu -2k+k\mu_k  \right)B[ e, P_2(e)  ]   \\
       &&- \left( (\lambda-\mu)(2k+1)+k(\lambda_k-\mu_k)  \right)   \left\langle kP_2(e),e^* \right\rangle P_2(e) .    
\end{eqnarray*}
  Taking in mind that $\lambda_k=k\lambda-k+1$ and $\mu_k=k\mu-k+1$, we complete the proof.   
\end{proof}

In addition, Theorem~\ref{lemm-compos} provides the connection between the Fekete--Szeg\"o mappings of inverses.

\begin{propo}\label{thm-compo}
  Let $f$  be represented by the series \eqref{series-2}. Then
$$\Psi_e(f^{-1},\lambda,\mu)=  -\Psi_e(f,2-\lambda,2-\mu).$$
In addition, denoting $g=f^{-1}$ we have for any $x\neq 0$,
\begin{equation}\label{R2R3-1}
\frac{1}{3!}D^3 g(0)[x^3]=-\|x\|^3 \cdot\Psi_e(f,2,2),\quad \mbox{where } e=\frac1{\|x\|}\,x.
\end{equation}
\end{propo}

\begin{proof} 
Represent $g$ by the series of homogeneous polynomials   $g(x)=x+Q_2(x)+Q_3(x)+\ldots$. Denote again $B=\frac12D^2f(0)$ and $C=\frac12D^2g(0)$.
Since $f\circ g= \Id$, it follows from \eqref{R2R3} that
\begin{equation}\label{R2R3-2}
Q_2(x)=-P_2(x), \quad  Q_3(x)=-2B[x,Q_2(x)]-P_3(x).
\end{equation}
Substituting $Q_2$ into $Q_3$ we get for any $e\in\partial\B$,
\begin{equation*}
Q_3(e)=-\left\{P_3(e)-D^2f(0)[e,P_2(e)]\right\}=-\Psi_e(f,2,2),
\end{equation*}
so \eqref{R2R3-1} holds.

In addition, in this case $\Psi_e(f\circ g,\lambda,\mu)=0$ for every $e\in\partial\B$ and any $\lambda,\mu\in\C$. Thus by \eqref{Psi-compo}, \eqref{R2R3-1} and  \eqref{R2R3-2} we have
\begin{eqnarray*}
\nonumber   0 &=&   \Psi_e(f,\lambda,\mu) + \Psi_e(g,\lambda,\mu) \\
\nonumber   &-& (\lambda-\mu) \left\{  \left\langle P_2(e) ,e^* \right\rangle Q_2(e) + \left\langle Q_2(e),e^* \right\rangle P_2(e) \right\}   \\
   &-&  \mu C[e,P_2(e)] + (2-\mu)B[e,Q_2(e)] ,
\end{eqnarray*}
or equivalently,
\begin{eqnarray*}
\Psi_e(g,\lambda,\mu) &\!\!=\!\!&
 (\lambda-\mu) \left\{  \left\langle P_2(e) ,e^* \right\rangle Q_2(e) + \left\langle Q_2(e),e^* \right\rangle P_2(e) \right\}   \\
   &\!\!+\!\!&  \mu C[e,P_2(e)] + (\mu-2)B[e,Q_2(e)] - \Psi_e(f,\lambda,\mu) \\
   &\!\! =\!\!& -\Psi_e(f,2-\lambda,2-\mu).
\end{eqnarray*}
The proof is complete.
\end{proof}

For $n\in\N$, it is natural to denote $f^{-n}=(f^{-1})^n$. Combining Propositions~\ref{thm-itera} and~\ref{thm-compo}, we get a general result: 

\begin{theorem}\label{thm-itera-gen}
Let  $f$ be represented by the series \eqref{series-2}. For $m\in\Z$  denote $T_{m,j}(x):=\frac{1}{j!}D^j(f^m)(0)[x^j]$. Then
  \begin{eqnarray*}
    T_{m,2}(x) &=& m P_2(x),  \\
    T_{m,3} (x) &=& m \|x\|^3\cdot \Psi_e(f,1-m,1-m),\quad \mbox{where } e=\frac1{\|x\|}\,x ,  \\
    \Psi_e(f^m,\lambda,\mu) &=& m \Psi_e(f,\lambda_m,\mu_m),
  \end{eqnarray*}
  where $\lambda_m=m\lambda-m+1$ and $\mu_m=m\mu-m+1$. 
\end{theorem}

As an immediate consequence of Theorem~\ref{thm-itera-gen} we get a new relation for the Fekete--Szeg\"o functional introduced in \cite{HKK2021}, see~\eqref{FS-previ}.

\begin{corol} For any $m\in\Z$, the following equality holds:
\[
\psi_e^{(2)}(f^m,\lambda)=m    \psi_e^{(2)}(f,m\lambda-m+1)  . 
\]
\end{corol}

\medskip

\subsection{Estimates for the Fekete--Szeg\"o mapping}\label{ssect-devia}

Let us denote 
$$R(f,g):= \Psi_e( f\circ g,\lambda,\mu) -\bigl(\Psi_e( f , \lambda, \mu)  +     \Psi_e( g, \lambda, \mu)\bigr). $$
Recall that by the addition formula (Corollary~\ref{corol-odd}), $R(f,g)=0$  when either $P_2(e)=0$  or $Q_2(e)=0$. Observe that otherwise $R(f,g)$ is the ``error term'' for the addition relation. In the case where the mappings $f$ and $g$ are of one-dimensional type, $\| R(f,g) \| = 2 |1-\lambda| \cdot |p_2(e)q_2(e)|$ by Corollary~\ref{corol-add}.

Now we estimate  the ``error term'' $R(f,g)$ in general. 

\begin{theorem}
Let $f$ and $g$ be represented by the series \eqref{series-2}. Denote $N_f:=\left\|\frac12 D^2f(0)\right\|$ and $N_g:=\left\|\frac12 D^2g(0)\right\|$.
Then $$\|R(f,g) \| \le \ell(\lambda,\mu)\cdot N_f\,N_g,$$
where $\ell (\lambda,\mu)=2 |\lambda-\mu|  + |\mu|   +   |\mu-2| $.
\end{theorem}
\begin{proof}
Let $B:=\frac12 D^2f(0)$ and  $C:=\frac12 D^2g(0)$ as above. 

We estimate the terms in the right-hand side in formula~\eqref{Psi-compo}:
\begin{eqnarray*}\label{QP1}
&& (1) \quad\|  \left\langle P_2(e) ,e^* \right\rangle  Q_2(e)\| = | \left\langle P_2(e) ,e^* \right\rangle|\cdot \|Q_2(e)\|\leq  \|P_2(e)\| \|Q_2(e)\| ;\\
&& (2) \quad\|\left\langle Q_2(e),e^* \right\rangle P_2(e) \| = |\left\langle Q_2(e),e^* \right\rangle |\cdot \| P_2(e)\|\leq  \|P_2(e)\| \|Q_2(e)\| ;\\
&& (3) \quad \|C[e,P_2(e)]\|\leq \|C\|\cdot\|e\|\cdot\|P_2(e)\| = N_g \|P_2(e)\|      ;\\
&& (4) \quad\|B[e,Q_2(e)]\|\leq \|B\| \cdot\|e\|\cdot\|Q_2(e)\| = N_f \|Q_2(e)\|  .
\end{eqnarray*}

By Theorem~\ref{lemm-compos}, for any $e\in\partial\B$ we have
\begin{eqnarray*}
 \nonumber   \Psi_e(f\circ g,\lambda,\mu)  &- & \bigl( \Psi_e( f , \lambda, \mu)  + 
                       \Psi_e(  g, \lambda, \mu)\bigr) \\
\nonumber  &=&-   (\lambda-\mu) \left[ \left\langle P_2(e) ,e^* \right\rangle        Q_2(e)  + \left\langle Q_2(e),e^* \right\rangle P_2(e) \right]    \\
&& - \mu C[e, P_2(e)]   -  (\mu-2) B[e,Q_2(e)] .
\end{eqnarray*}

Substitute estimates (1)--(4) and get 
\begin{eqnarray*}
 && \| \Psi_e(f\circ g,\lambda,\mu)- \Psi_e( f , \lambda, \mu) -\Psi_e(  g, \lambda, \mu) \| \\
& &\leq  2 |\lambda-\mu| \|P_2(e)\| \|Q_2(e)\|  
+ |\mu|  N_g \|P_2(e)\|     +   |\mu-2|  N_f \|Q_2(e)\| .
\end{eqnarray*}
Since $P_2(e)=B(e,e)$ and $Q_2(e)=C(e,e)$, the last estimate implies the desired result. 
\end{proof}

\vspace{1mm}

Another issue is to estimate  $\Psi_e(\cdot,\lambda,\mu) $ over the class of bounded mappings. We do it for mappings  of one-dimensional type.

\begin{theorem}\label{propo-1dim}
Let $f\in\Hol(\B,X)$ be a bounded normalized mapping of one-dimensional type, and $M:=\sup_{x\in\B}\|f(x)\| >1$. 
Then
\[
\left\| \Psi_e(f,\lambda,\mu) \right\|  \le \frac{M^2-1}M \max\left\{ 1, \left| \frac{(M^2 -1)\lambda +1}{M}   \right|  \right\}.
\]
\end{theorem}
\begin{proof}
By our assumption there exists $s \in\Hol(\B,\C)$ such that  $f(x)=s(x)x$ with $s(x)=1+p_1(x)+p_2(x)+\ldots,$ and $M=\sup_{x\in\B}|s(x)|.$  Denote $\widehat{s}(x):=\frac1M s(x)$ and $\widetilde{s}(x):=\displaystyle\frac{\frac1M -\widehat{s}(x)}{1- \frac1M  \widehat{s}(x)}= \frac{M(1- s(x))} {M^2-s(x)}$. Then $\widetilde{s}$ can be represented by the series of homogeneous polynomials $\displaystyle \widetilde{s}(x)=\sum_{j=1}^\infty \widetilde{p}_j(x),$ where
\[
\widetilde{p}_1(x) = -\frac{M}{M^2-1}p_1(x) \quad \mbox{and} \quad \widetilde{p}_2(x) = - \frac{M}{M^2-1} \left( \frac{1}{M^2-1}p_1^2(x)+ p_2(x) \right)\!.
\]
Given $\gamma\in\C$, consider the homogeneous polynomial of order two
\[
T(x)=\widetilde{p}_2(x) -\gamma\left( \widetilde{p}_1(x) \right)^2 = \frac{-M}{M^2-1} \left( p_2(x) + \frac{1+M\gamma} {M^2-1}\, p_1(x)^2   \right) \! .
\]

According to \cite[Theorem~2.1]{D-L-21}, we have
\[
\|T\| := \sup_{\|x\|=1} |T(x)| \le \max\left\{ 1,|\gamma|\right\}.
\]
Denote $\lambda=-\frac{1+M\gamma}{M^2-1}$. Then 
\[
\frac{M}{M^2-1}\left| p_2(x) - \lambda p_1^2(x)   \right| \le \max\left\{ 1, \left| \frac{(M^2-1)\lambda +1}{M}  \right| \right\}\!.
\]
Since $\Psi_e(f,\lambda,\mu)=\Psi_e(f,\lambda)=\left[ p_2(e) - \lambda p_1^2(e) \right]e$, the result follows.
\end{proof}

\medskip

\section{Fekete--Szeg\"o mapping and class $\m$}\label{sect-m}
\setcounter{equation}{0}

In this section we study the Fekete--Szeg\"o mapping $\Psi_e(\cdot,\lambda,\mu)$ over the class  
\begin{eqnarray*}
 \m\!  : =\! \left\{h\in\A(\B)\!: \Re \left\langle h(x), x^* \right\rangle\ge0 \quad\mbox{for all } x\in\B\!\setminus\!\!\{0\},\ x^*\in J(x) \right\}
\end{eqnarray*}
(see, for example, \cite{GK2003} and references therein). We assume that $h\in\m$ is represented by the series of homogeneous polynomials  
\begin{equation}\label{series-3}
  h(x)=x+\sum_{k=2}^{\infty} H_k(x).
  \end{equation}

An important property of the class $\m$ is that it consists of  {\it semigroup generators} in the open unit ball $\B$. This means that the initial value problem 
\begin{equation}\label{sem}
\left\{  \begin{array}{l}
           \displaystyle\frac{\partial u_t(x)}{\partial t} +h(u_t(x))=0, \vspace{2mm} \\
           u_0(x)=x 
         \end{array}                \right.
\end{equation} 
has the unique solution $u_t(x) \in\B$ for all $x \in\B$ and $t\ge 0$. The family $\left\{u_t\right\}_{t\ge0}\subset\Hol(\B,\B)$ forms the semigroup of holomorphic self-mappings of~$\B$, see \cite{E-R-S04, E-R-S-19, R-Sbook}.

Given $t>0$ consider the semigroup element $u_t$. Being a holomorphic mapping, it can be represented by the series of homogenous polynomials. 
We prove that the polynomial of degree three in this representation is proportional to the Fekete--Szeg\"o mapping of $h$.

 \begin{theorem}\label{th-semigr}
Let $u$ be the semigroup generated by $h\in\m$. Then for every $e\in\partial\B$ and every $t\ge0$ we have
\begin{eqnarray*}
 \frac1{2!} D^2u_t(0)[e^2] &=& \exp(-t)(\exp(-t)-1)\cdot H_2(e),\\
\frac{1}{3!}D^3u_t(0)[e^3] &=& \frac{ \exp(-t)(\exp(-2t)-1)}{2}\cdot\Psi_e(h, \lambda, \lambda),
\end{eqnarray*}
where $\lambda=\frac{2(1-\exp(-t))}{1+\exp(-t)} $.
 \end{theorem}
 Recall that $\Psi_e(f, \lambda, \lambda)=\Psi_e^{(1)}(f,\lambda)$, see formulae \eqref{FS-previ} and \eqref{FS-oper1}.  
 \begin{proof}
It is known that $Du_t(0)[x]=\exp(-t)$ for any $t>0$. Represent $u_t$ by the series $ u_t(x)=\exp(-t) \left[ x  + \sum\limits_{k=2}^{\infty} S_k(t,x)\right],$  
 where $S_k, k\ge2,$ is a homogeneous polynomial of degree $k$. Substitute the series for $u_t$ and $h$ in \eqref{sem} and get
  \begin{eqnarray*}
     &&- \exp(-t) \left[ x  + S_2(t,x) + S_3(t,x) \right]  + \exp(-t) \left[ \frac{\partial}{\partial t}S_2(t,x) +\frac{\partial}{\partial t} S_3(t,x) \right]   \\
     &&+ \exp(-t) \left[ x  + S_2(t,x) + S_3(t,x) \right] + \frac{1}{2!} D^2 h(0)\left[ \left( \exp(-t)(x+S_2(t,x))  \right)^2  \right] \\
     &&  +\frac1{3!} D^3 h(0)\left[ \left( \exp(-t)x \right)^3  \right]= o(\|x\|^3) . 
  \end{eqnarray*}
  This enables us (because every Fr\'{e}chet derivative of order $k$ is a bounded symmetric $k$-linear operator)  to equate terms of the same degree. So,
  \begin{eqnarray*}
     &&  \exp(-t)  \frac{\partial}{\partial t}S_2(t,x)   + \frac{1}{2!} \exp(-2t) D^2 h(0)\left[ x^2  \right] =0,\\
     &&  \exp(-t) \frac{\partial}{\partial t} S_3(t,x)  + \exp(-2t) D^2 h(0)\left[ x, S_2(t,x) \right]\\
     && \hspace{4cm} +\frac1{3!}\exp(-3t) D^3 h(0)\left[ x ^3  \right]= 0 .
  \end{eqnarray*} 
  The solution of the first equation is $S_2(t,x)= \frac{1}{2!} (\exp(-t)-1) D^2 h(0)\left[ x^2  \right] $. Substitute it to the second one:
   \begin{eqnarray*}
     \frac{\partial}{\partial t} S_3(t,x)  + \exp(-t)(\exp(-t)-1) D^2 h(0)\left[ x, \frac{1}{2!} D^2 h(0)\left[ x^2  \right] \right] \\
     +\frac1{3!}\exp(-2t) D^3 h(0)\left[ x ^3  \right]= 0 .
   \end{eqnarray*} 
   Therefore, 
   \begin{eqnarray*}
      S_3(t,x) &=& \frac12(\exp(-2t)-1) \frac{1}{3!} D^3 h(0)[x^3] \\
      &+& (\exp(-t)-1)^2\cdot\frac14D^2 h(0)\left[x,D^2h(0)[x^2]\right] \\
       &=&  \frac12(\exp(-2t)-1) \left[ H_3(x) - \frac{1-\exp(-t)}{1+\exp(-t)} D^2 h(0)\left[x,H_2(x)\right]  \right],
   \end{eqnarray*}
   which completes the proof.   
 \end{proof}
 
To explain another important feature of this class,  recall that a  biholomorphic mapping $f\in \A(\B)$ is said to be starlike (see \cite{E-R-S04, GK2003, R-Sbook, STJ-77})  if $tw\in f(\B)$ for every $w\in f(\B)$ and $t\in[0,1]$. The set of all starlike mappings is denoted by $\Sta$.
The following fact is well known (see Proposition~2.5.3 and corresponding references in \cite{E-R-S04}).
\begin{propo}\label{propo-star}
  Let $f\in\Hol(\B,X)$ be a normalized and locally biholomorphic mapping. Then $f\in\Sta$  if and only if there is a unique mapping $h\in\m$ such that $Df(x)[h(x)]=f(x)$  in $\B$.
\end{propo}

It turns out that the Fekete--Szeg\"o mappings of $f$ and $h$ can be expressed one by another. Moreover, this allows us to obtain sharp estimates on the Fekete--Szeg\"o mapping over the class $\m$.

 \begin{theorem}\label{thm-classM}
   Let $h\in\m$ and $f$ be the unique normalized solution to the equation $Df(x)[h(x)]=f(x)$. Then for any $e \in \partial \B$ and $\lambda, \mu \in \C$ we have
      \begin{equation}\label{Psi-f-h}
\Psi_e\left(h,2\lambda,2\mu\right)=-2\Psi_e(f,1-\lambda,1-\mu).
 \end{equation}    
 Moreover, 
     \begin{equation}\label{FSfor-h}
\left|\left\langle\Psi_e\left(h,\lambda,0\right),e^*\right\rangle\right| \le 2\max\left\{1 ,\left|2\lambda-1\right| \right\} \quad \text{for any } \lambda \in \C.
 \end{equation}
  In addition, if the homogeneous polynomials $H_2$ and $H_3$ in \eqref{series-3} are of one-dimensional type, then 
\begin{equation}\label{FSfor-h1}
\left\|\Psi_e\left(h,\lambda,\mu\right)\right\| \le 2\max\left\{1 ,\left|2\lambda-1\right| \right\} \quad \text{for any } \lambda, \mu \in \C.
 \end{equation}
 Both estimates are sharp.
 \end{theorem}

\begin{proof}
Recall that for every $h\in\m$ there exists the unique mapping  $f\in \Sta$ such that $Df(x)[h(x)]=f(x)$  in $\B$ by \cite[Proposition~3.7.5]{E-R-S04}. 

Let $f$ and $h$ be represented by the series \eqref{series} and \eqref{series-3}, respectively.  
Fix $e \in \partial \B$. It was shown in the proof of Theorem~3.1 in \cite{HKK2021} that 
 \begin{eqnarray}\label{H23-P23}
\begin{array}{l}
   \displaystyle  P_2(e)=-\frac{1}{2!}D^2h(0)[e^2] \qquad \text{ and }\vspace{3mm} \\
\displaystyle P_3(e)=\frac{1}{3!}\left(  -\frac{1}{2}D^3h(0)[e^3]+\frac{3}{2}D^2h(0)\left[e,D^2h(0)[e^2] \right]\right)\!.
\end{array}
 \end{eqnarray} 
For any $\lambda,\mu\in\C,$  and $e^* \in J(e)$, Definition~\ref{def-FS-oper} implies
\begin{eqnarray*}
 \Psi_e(f,\lambda,\mu)&=& P_3(e)       - \frac\mu2 D^2f(0)\left[ e, P_2(e) \right] 
           -  (\lambda-\mu) \left\langle P_2(e),e^* \right\rangle P_2(e)\\
 &=& \frac{1}{3!} \left(-\frac{1}{2}D^3h(0)[e^3]+\frac{3}{2}D^2h(0)\left[e,D^2h(0)[e^2] \right]\right)\\
 &&+    \frac{\mu}{2}D^2h(0)\left[e,-\frac{1}{2!}D^2h(0)[e^2] \right] \\
 &&- (\lambda - \mu) \left\langle -\frac{1}{2!}D^2h(0)[e^2], e^* \right\rangle  \left( -\frac{1}{2!}D^2h(0)[e^2] \right)\!.
  \end{eqnarray*}
Consequently,
\begin{eqnarray*}
 2\Psi_e(f,1-\lambda,1-\mu) &=& -\frac{1}{3!} D^3h(0)[e^3]+2\mu\frac{1}{2!}D^2h(0)\left[e,\frac{1}{2!}D^2h(0)[e^2] \right]\\
&&-2 (\mu-\lambda) \left\langle \frac{1}{2!}D^2h(0)[e^2], e^*  \right\rangle\frac{1}{2!} D^2h(0)[e^2]\\
&=&-\Psi_e\left(h,2\lambda,2\mu\right).
  \end{eqnarray*}
  So, the first statement follows. 

Because $f\in\A(\B)$ is a starlike mapping, the estimate 
   \[
   \left| \psi_e^{(2)}(f,\lambda) \right|  \le \max\left\{1 ,\left|4\lambda-3\right| \right\}
   \]  
holds and is sharp by  \cite[Corollary 3.4]{HKK2021}. Here $ \psi_e^{(2)}(f,\lambda)   =  \left\langle \Psi_e(f,\lambda,1) ,e^*\right\rangle$, see~\eqref{FS-previ}.

Thanks to the proved equality \eqref{Psi-f-h}, we have  
\begin{equation*}
\left|\left\langle\Psi_e\left(h,2\lambda,0\right),e^*\right\rangle\right|=2\left|\left\langle\Psi_e(f,1-\lambda,1),e^*\right\rangle\right| \le 2 \max\left\{1 ,\left|1-4\lambda\right| \right\}.
 \end{equation*}
 Thus estimate \eqref{FSfor-h} holds and is sharp.

If the homogeneous polynomial $H_2$ and $H_3$ are of one-dimensional type, one concludes that estimate \eqref{FSfor-h1} holds and is sharp by assertion (B) in Lemma~\ref{lemm-F-S-norm}.
\end{proof}

\begin{remar}\label{rem1a}
 It is worth mentioning that formulas~\eqref{H23-P23} imply that the polynomials $P_2$ and $P_3$
 are of one-dimensional type if and only if $H_2$ and $H_3$~are. 
\end{remar}

For a given point $e\in\partial\B$,  let us denote by $\m_e$  the subclass of $\m$ consisting of mappings satisfying the estimate \eqref{FSfor-h1} for all $\lambda,\mu\in\C$. Also denote by $\St$ the class of biholomorphic mappings $f$ such that $(Df)^{-1}[f]\in\m_e$. 
Observe that the classes $\m_e$ and $\St$ are not empty. Indeed, if the homogeneous polynomials of the second and third order of a mapping $h\in \m$ (of $f\in\Sta$) are of one-dimensional type, then by Theorem~\ref{thm-classM}, $h\in\m_e$ and $f\in\St$. 
  
\begin{corol}
  Let $f\in\St$. Then 
  \begin{equation}\label{FS-for-Ste}
 \left\| \Psi_e(f,\lambda,\mu) \right\|  \le \max\left\{1 ,\left|4\lambda-3\right| \right\}.
  \end{equation}
 \end{corol}
Obviously  $\St\subset\Sta$. Thanks to Lemmas~\ref{lemm-F-S-norm1} and~\ref{lemm-F-S-norm} and Remark~\ref{rem1a}, this corollary partially generalizes results in \cite{La-X-21, E-J-22a}.    In this connection, it is natural to ask 
\begin{itemize}
  \item [$*$] Does $\St$ consist of mappings such that the homogeneous polynomials of the second and third order in their expansion are of one-dimensional type only?
  \item [$*$] If not, does the relation $\St=\Sta$ hold?
\end{itemize}

In the next examples we consider the two-dimensional Euclidian space $X=\C^2$ and show that the answer to both questions is negative.

\begin{examp}
Let $f\in\Hol(\B,X)$ be defined by
\begin{equation}\label{f1}
f(x_1,x_2) = \left( x_1-\frac{x_1^2}{2+x_1} , x_2 -\frac{x_1^2}{2+x_2}   \right)\! .
\end{equation}
In order to justify that $f\in\Sta$, we first find the mapping $h:=(Df)^{-1}f$ by a direct computation:
\begin{equation*}
h(x_1,x_2) = \left( x_1+\frac{x_1^2}{2} , x_2+ \frac{x_1^2}{2}   \right)\! .
\end{equation*}
Then
\[
\Re \left\langle h(x), x \right\rangle \ge \|x\|^2 - \frac{|x_1|^2(|x_1|+|x_2|)}{2}.
\]
One can easily see that this expression is non-negative in $\B$. Hence $h \in \m$. Therefore $f$ is starlike by Proposition~\ref{propo-star}. 

It follows from \eqref{f1} that $P_2(x)=-\frac{x_1^2}{2}\left(1,1\right) $  and $P_3(x)= \frac{x_1^3}{4}\left(1,1\right) $. Meanwhile, this implies $\frac{1}{2} D^2f(0)[u,v] =-\frac{u_1v_1}{2}(1,1).$ Clearly, polynomials $P_2$ and $P_3$ are not of one-dimensional type. At the same time, taking $e=(1,0)$ we get
\[
\Psi_e(f,\lambda,\mu)=\frac{1}{4}\left(1,1\right)  -\frac{\mu}{4}\left(1,1\right) -\frac{\lambda-\mu}{4} \left(1,1\right) =\frac{1-\lambda}{4}\left(1,1\right),
\]
so that $\left\| \Psi_e(f,\lambda,\mu) \right\|=\frac{|1-\lambda|}{2\sqrt2}$ and estimate \eqref{FS-for-Ste} is satisfied. Thus $f\in\St$.
\end{examp}

\begin{examp}
Now we consider the mapping $f\in\Hol(\B,X)$  defined by
\begin{equation*}\label{f2}
f(x_1,x_2) = \left( x_1  , x_2 +(x_1+x_2)\left(e^{-x_1} - 1\right)  \right)\! .
\end{equation*}
Similarly to the previous example one checks that $f\in\Sta$. Its homogeneous polynomials of the second and third order are $P_2(x)=\left(0, -x_1(x_1+x_2) \right)$ and $P_3(x)= \left(0, \frac{x_1^2}2(x_1+x_2) \right)$. 

This leads to $\frac{1}{2} D^2f(0)[u,v] =-\left(0,u_1v_1 + \frac{u_1v_2+u_2v_1}{2}\right).$  Take again $e=(1,0)\in\partial\B$. Then
\[
\Psi_e(f,\lambda,\mu)=\left(0,\frac{1}{2} \right)  -\left(0, \frac{\mu}{2}\right)  - (\lambda-\mu)\cdot0\cdot \left(0,-1\right) =\left(0,\frac{1-\mu}{2}\right)\!.
\]
Therefore $\left\| \Psi_e(f,\lambda,\mu) \right\|=\left| 1+\frac\mu2 \right|$. So, inequality \eqref{FS-for-Ste} holds or fails depending on the relationship between $\lambda$ and $\mu$. Thus  $f\in\Sta\setminus\St $.
\end{examp}

\end{document}